\documentclass[reqno]{amsart}
\usepackage{amsmath,amsthm, mathrsfs, verbatim, hyperref}
\hypersetup{linkcolor=blue, citecolor=red, colorlinks=true,
pdfauthor={Rintaro Yoshida}}

\def \f12{\frac 12}

\def\jacw8{W_{\alpha,\beta}}

\newcommand{\disp}{\displaystyle}

\newcommand{\scr}[1]{\mathscr #1}

\newcommand{\abs}[1]{\left\vert#1\right\vert}
\newcommand{\set}[1]{\left\{#1\right\}}

\newcommand{\lr}[1]{\left(#1\right)}
\newcommand{\br}[1]{\left[#1\right]}

\newcommand{\nats}{\mathbb N}

\newcommand{\reals}{\mathbb R}
\newcommand{\complex}{\mathbb C}

\theoremstyle{plain}
\newtheorem{thm}{Theorem}
\newtheorem{lem}[thm]{Lemma}
\newtheorem{prob}{Open Problem}

\newtheorem{prop}[thm]{Proposition}

\theoremstyle{remark}

\theoremstyle{definition}
\newtheorem{defn}[thm]{Definition}

\theoremstyle{definition}
\newtheorem{xmpl}[thm]{Example}

\theoremstyle{definition}
\newtheorem{notn}[thm]{Notation}

\begin{document}

\title{On some questions of Fisk and Br\"and\'en}
\author{Rintaro Yoshida}
\address{Department of Mathematics, University of Hawaii, Manoa, Honolulu, HI 96822}
\email{yoshi@math.hawaii.edu}

\date{\today}

\begin{abstract}
P. Br\"and\'en recently proved a conjecture due to S. Fisk, R. P. Stanley, P. R. W. McNamara and B. E. Sagan. In addition,  P. Br\"and\'en gave a partial answer to a question posed by S. Fisk regarding the distribution of zeros of polynomials under the action of certain non-linear operators. In this paper, we give an extension to a result of P. Br\"and\'en, and we also answer a question posed by S. Fisk.
\end{abstract}


\maketitle


\section{Introduction}
\label{s:intro}

P. Br\"and\'en \cite{B09} recently proved a conjecture due to S. Fisk, R. P. Stanley, P. R. W. McNamara and B. E. Sagan. 

\smallskip

\begin{thm}[P. Br\"and\'en \cite{B09}]
\label{stly}
If a real polynomial $\sum_{k=0}^n a_k x^k$ has only real negative zeros, then the associated polynomial $\sum_{k=0}^n (a_k^2-a_{k-1}a_{k+1}) x^k$, also has only real negative zeros, where $a_{-1}=a_{n+1}=0$.
\end{thm}

\smallskip

S. Fisk \cite{Fisk} posed a problem related to Theorem \ref{stly}, which may be formulated as follows.

\begin{prob}
\label{fisk}
Let $r\in\nats$. If a real polynomial $\sum_{k=0}^n a_k x^k$ has only real negative zeros, then does the associated polynomial $\sum_{k=0}^n (a_k^2-a_{k-r}a_{k+r}) x^k$, where $a_{t}=a_{s}=0$ for $t<0$ and $s>n$, have only real negative zeros?
\end{prob}

\smallskip

In \cite{B09}, P. Br\"and\'en provides an affirmative answer to Open Problem \ref{fisk} for $r=0,1,2,3$. In Section ~\ref{s:main} of this note, we will extend the aforementioned result of P. Br\"and\'en, and we give a complete answer to Open Problem \ref{fisk}. We investigate, in Section ~\ref{s:related}, a problem of S. Fisk \cite{Fisk} involving the minors of a particular matrix related to Theorem \ref{stly}. Subsequently, in Section \ref{s:app}, we apply the methods discussed in Sections \ref{s:main} and \ref{s:related} to another non-linear operator which arises from an inequality introduced by D. K. Dimitrov. We conclude this paper with an example and several open problems.

\section{Preliminaries}
\label{s:setup}

\begin{defn}
\label{lp} A real entire function $\varphi$ is said to belong to the \emph{Laguerre-P\'olya class}, denoted $\scr  L$-$\scr P$, if $\varphi$ is the uniform limit on compact subsets of $\complex$, of real polynomials with only real zeros. A function $\varphi\in\scr  L$-$\scr P\cup \set{0}$ if and only if it can be expressed as
$$\varphi(x)= c\, x^n e^{-ax^2+bx} \prod_{j=1}^\infty (1+\rho_jx) e^{-\rho_jx},$$
where $b,c,\rho_j\!\in\reals,\, a\ge0,\; j,n\!\in\nats$, and $\sum_{j=1}^\infty \rho_j^2<\infty$ (see \cite[Chapter VIII]{levin}). The subclass $\scr  L$-$\scr P^+ \subset \scr  L$-$\scr P$, consists of those functions $\varphi\in\scr  L$-$\scr P$ that have non-negative Taylor coefficients. 
\end{defn}

The class $\scr  L$-$\scr P^+$ has the following characterization (see \cite[Chapter VIII]{levin}).

\begin{thm}
\label{lp+}
A real entire function $\varphi\in\scr  L$-$\scr P^+\cup \set{0}$ if and only if its Hadamard product representation can be expressed in the form
$$\varphi(x)= c x^m e^{ax}\prod_{j=1}^\infty(1+\rho_jx),$$
where $ a,c,\rho_j\ge 0,\; m,j\in\nats$, and $\;\sum_{j=1}^\infty \rho_j<\infty$.
\end{thm}

\smallskip

\begin{defn}
We will refer to the following notation frequently in the sequel. We define
\begin{equation}
\label{eq:p+}
\scr  L\text{-}\scr P_\nats^+:= \reals[x]\cap\scr  L\text{-}\scr P^+.
\end{equation}

\label{ms}
\noindent A sequence of real numbers $\set{\gamma_k}_{k=0}^\infty$ is a \textit{multiplier sequence},\; if $n\in\nats$, and $\sum_{k=0}^n a_k x^k=p(x)\in\scr  L\text{-}\scr P_\nats^+$, then $\sum_{k=0}^n \gamma_k a_k x^k\in\scr  L\text{-}\scr P_\nats^+\cup \set{0}$.
\end{defn}

\smallskip

Multiplier sequences were characterized in a seminal paper by G. P\'olya and J. Schur \cite{PS}.

\begin{thm}[G. P\'olya and J. Schur \cite{PS}]
 \label{PS}
 Let $\set{\gamma_k}_{k=0}^\infty$ be a sequence of real numbers, and let $T: \reals[x]\to \reals[x]$ be the corresponding (diagonal) linear operator defined by $T[x^k]=\gamma_kx^k$, for $k\in\nats$. Define $\varphi(x)=T[e^x]$ to be the formal power series
$$\varphi(x)=\sum_{k=0}^\infty \frac{\gamma_k}{k!}x^k.$$
Then the following are equivalent:

\begin{enumerate}
\item  $\set{\gamma_k}_{k=0}^\infty$ is a multiplier sequence;
\item  $T[\scr L$-$\scr P]\subseteq \scr L$-$\scr P\cup \set{0}$;
\item $\varphi(x)$ is the uniform limit of polynomials with only real zeros of the same sign on compact subsets of $\complex$;
\item Either $\varphi(x)$ or $\varphi(-x)$ is an entire function that can be written as
$$Cx^ne^{ax}\prod_{k=1}^\infty (1+\alpha_k x),$$\\
where $n\in\nats, C\in\reals, a, \alpha_k\ge0$ for all $k\in\nats$ and $\sum_{k=1}^\infty \alpha_k<\infty$;
\item For all non-negative integers $n$ the polynomial $T[(1+x)^n]$ has only real zeros of the same sign.
\end{enumerate}

\end{thm}

\smallskip

A particularly useful multiplier sequence given in the following example will be frequently referred to in the sequel.

\begin{xmpl}
\label{xms}
For each $\mu>0$, $\set{1/\Gamma(k+\mu)}_{k=0}^\infty$ is a multiplier sequence (see \cite[Theorem 4.1]{CCkluwer}). In particular, for $j\in\nats$, $\set{1/(k+j)!}_{k=0}^\infty$ is a multiplier sequence.
\end{xmpl}

\begin{lem}
\label{rvf}
For $n,d\in\nats$, the sequence $\set{\frac{1}{(n-k+d)!}}_{k=0}^\infty$ $($where $\frac{1}{k!}=0$ for $k<0)$ is a multiplier sequence.
\end{lem}

\begin{proof}
Fix $n\in\nats$. By Theorem \ref{PS} (v), it suffices to show that the polynomial $f(x):=\sum_{k=0}^n \binom{n}{k} \frac{1}{(n-k+d)!} x^k$ has only real negative zeros. The sequence $\set{1/(k+d)!}_{k=0}^\infty$ is a multiplier sequence, thus $p(x) := \sum_{k=0}^n \binom{n}{k} \frac{1}{(k+d)!} x^{k} = \sum_{k=0}^n \binom{n}{k} \frac{1}{(n-k+d)!} x^{n-k}$ has only real negative zeros. Thus it follows that $f(x)$ has only real negative zeros, as desired.
\end{proof}

\smallskip

The following notation follows P. Br\"and\'en \cite[Section 4]{B09}.

\begin{defn}\label{def1} 
Let $\alpha=\set{\alpha_k}_{k=0}^\infty$ be a fixed sequence of complex numbers and given a finite sequence, $\set{a_k}_{k=0}^n$, define two new sequences $\set{b_k(\alpha)}_{k=0}^\infty$ and $\set{c_k(\alpha)}_{k=0}^\infty$, where 
$$b_k(\alpha):=\sum_{k=0}^\infty \alpha_j a_{k-j}a_{k+j} \quad\text{ and }\quad c_k(\alpha):=\sum_{k=0}^\infty \alpha_j a_{k-j}a_{k+1+j},$$
and $a_j=0$ if $j\notin\set{0,1,\ldots, n}$. Also define two non-linear operators acting on polynomials, $U_\alpha, V_\alpha: \complex[x]\to\complex[x]$, by
\begin{equation}
\label{eq:uava}
U_\alpha\lr{\sum_{k=0}^n a_k x^k}:=\sum_{k=0}^n b_k(\alpha)x^k\quad \text{ and }\quad V_\alpha\lr{\sum_{k=0}^n a_k x^k}:=\sum_{k=0}^n c_k(\alpha)x^k.
\end{equation}
\end{defn}

\smallskip

P. Br\"and\'en extends the non-linear operators $U_\alpha$ and $V_\alpha$ from $\scr  L\text{-}\scr P_\nats^+$ to \mbox{$\scr L$-$\scr P^+$} (cf. Definition \ref{ms}), in order to prove the following theorems \cite[Theorem 5.7, Theorem 5.8]{B09}.

\begin{thm}[P. Br\"and\'en \cite{B09}]
\label{B}
If $\alpha=\set{\alpha_k}_{k=0}^\infty$ is a sequence of real numbers, then the following are equivalent.

\smallskip

\begin{enumerate}
\item $U_\alpha[\scr  L\text{-}\scr P_\nats^+] \subseteq \scr  L\text{-}\scr P_\nats^+ \cup \set{0}$.
\item $U_\alpha[e^x]\in \scr L$-$\scr P^+\cup \set{0}$; that is, 
$$\sum_{k=0}^\infty\lr{\sum_{j=0}^k \frac{\alpha_j}{(k+j)! (k-j)!}}x^k \in \scr L\text{-}\scr P^+\cup \set{0}.$$
\item $U_\alpha[\scr L\text{-}\scr P^+]\subseteq  \scr L\text{-}\scr P^+\cup \set{0}$.
\end{enumerate}
\end{thm}

\begin{thm}[P. Br\"and\'en \cite{B09}]
\label{B2}
If $\alpha=\set{\alpha_k}_{k=0}^\infty$ is sequence of real numbers, then the following are equivalent.

\smallskip

\begin{enumerate}
\item $V_\alpha[\scr  L\text{-}\scr P_\nats^+] \subseteq \scr  L\text{-}\scr P_\nats^+ \cup \set{0}$.
\item $V_\alpha[e^x]\in \scr L$-$\scr P^+\cup \set{0}$; that is, 
$$\sum_{k=0}^\infty\lr{\sum_{j=0}^k \frac{\alpha_j}{(k+1+j)! (k-j)!}}x^k \in \scr L\text{-}\scr P^+\cup \set{0}.$$
\item $V_\alpha[\scr L\text{-}\scr P^+]\subseteq  \scr L\text{-}\scr P^+\cup \set{0}$.
\end{enumerate}
\end{thm}

\smallskip

\section{Main Results}
\label{s:main}

Recall the definition of the operators $U_\alpha$ and $V_\alpha$ (see~\eqref{eq:uava}). For $r\in\nats$, define 
\begin{equation}
\label{eq:sr}
S_r:=U_\alpha \quad \text{ and }\quad \widetilde{S}_r:=V_\alpha,
\end{equation}
where $\alpha=\set{\alpha_k}_{k=0}^\infty$, $\alpha_0=1$, $\alpha_r=-1$, and $\alpha_k=0$ if $k\notin\set{0,r}$. With the aid of Theorem \ref{PS}, Theorem \ref{B}, and Theorem \ref{B2}, P. Br\"and\'en proved $S_r[\scr  L\text{-}\scr P_\nats^+] \subseteq \scr  L\text{-}\scr P_\nats^+$ \cite[Proposition 6.2]{B09} and $\widetilde{S}_r[\scr  L\text{-}\scr P_\nats^+] \subseteq \scr  L\text{-}\scr P_\nats^+$ \cite[Proposition 6.3]{B09} for $r=0,1,2,3$, where $\scr  L\text{-}\scr P_\nats^+$ is defined in~\eqref{eq:p+}. The following propositions are extensions of the aforementioned results of P. Br\"and\'en.

\begin{prop}
$S_4[\scr  L\text{-}\scr P_\nats^+]\subseteq \scr  L\text{-}\scr P_\nats^+\cup \set{0}$, where $S_4$ is defined in~\eqref{eq:sr}. 
\end{prop}

\begin{proof}
By definition, $S_4[e^x]=\sum_{k=0}^\infty [a_{k}^2-a_{k-4}a_{k+4}]x^k$, where $a_k=1/k!$, and $a_k=0$ for $k<0$. By Theorem \ref{B}, it suffices to show that 
$$\disp S_4[e^x]=\sum_{k=0}^\infty \frac{8(2 k + 1) (k^2 + k + 3)}{k!(k+4)!}x^k\in \scr L\text{-}\scr P^+.$$
To this end, consider
$$\indent\disp f(x):=\sum_{k=0}^\infty \frac{8(2 k + 1) (k^2 + k + 3)(5 + k) (6 + k) (7 + k)}{k!}x^k  = p(x)e^x,$$
where the polynomial 
$$p(x)=5040 + 35280 x + 52920 x^2 + 29400 x^3 + 6360 x^4 + 552 x^5 + 16 x^6$$ 
has only real negative zeros. This assertion can be verified by using Mathematica in conjunction with the intermediate value theorem. Thus the entire function $f(x)\in\scr L\text{-}\scr P^+$,
and by Theorem \ref{PS}, the sequence 
$$\set{8(2 k + 1) (k^2 + k + 3)(5 + k) (6 + k) (7 + k)}_{k=0}^\infty$$ 
is a multiplier sequence. Next we apply the multiplier sequence $\set{1/(k+7)!}_{k=0}^\infty$  (cf. Example \ref{xms}) to the entire function $f(x)$ to obtain

\smallskip

$\disp \sum_{k=0}^\infty \frac{8(2 k + 1) (k^2 + k + 3)(5 + k) (6 + k) (7 + k)}{k!(k+7)!}x^k$ 
$$=\sum_{k=0}^\infty \frac{8(2 k + 1) (k^2 + k + 3)}{k!(k+4)!}x^k=S_4[e^x] \in \scr L\text{-}\scr P^+$$
by Theorem \ref{PS}. \end{proof}

\smallskip

A similar argument, \emph{mutatis mutandis}, establishes the following proposition.

\begin{prop}
$\widetilde{S}_4[\scr  L\text{-}\scr P_\nats^+]\subseteq \scr  L\text{-}\scr P_\nats^+\cup \set{0}$, where $\widetilde S_4$ is defined in~\eqref{eq:sr}. 
\end{prop}

\smallskip

Next, we consider Open Problem \ref{fisk} and a related question posed by P. Br\"and\'en \cite[p.11]{B09}, both of which fail. In particular, we will show that \mbox{$\,S_6[e^x],\, \widetilde{S}_6[e^x] \not\in\scr L$-$\scr P^+$}, where
\begin{equation}
\label{eqs6}
S_6[e^x]=\sum_{k=0}^\infty [a_{k}^2-a_{k-6}\,a_{k+6}]x^k, 
\end{equation}
and
\begin{equation}
\label{eqwts6}
\widetilde S_6[e^x]=\sum_{k=0}^\infty [a_{k}a_{k+1}-a_{k-6}\,a_{k+7}]x^k, 
\end{equation}
$\text{ where }a_k=1/k!, \text{ and }a_k=0 \text{ for }k<0$. Thus, by Theorem \ref{B} and Theorem \ref{B2}, $S_6[\scr  L\text{-}\scr P_\nats^+]\not\subseteq \scr  L\text{-}\scr P_\nats^+$ and $\widetilde S_6[\scr  L\text{-}\scr P_\nats^+]\not\subseteq \scr  L\text{-}\scr P_\nats^+$, where $\scr  L\text{-}\scr P_\nats^+$ is defined in~\eqref{eq:p+}. 

\smallskip

\begin{lem}
\label{s6error}
Set $f(x):=S_6[e^x]=\sum_{k=0}^\infty b_{k}x^k$, its partial sum $f_{n}(x):=\sum_{k=0}^n b_{k}x^k$, and $E_n(x):=f(x)-f_n(x)$. If $x_0=-43$, then 
$$|E^{(j)}_{30}(x_0)|<5\times10^{-18},$$
where $E_n^{(j)}(x)$ denotes the $j$-th derivative for $j=0,1,2$.
\end{lem}

\begin{proof}
The infinite sum obtained by the power series $f(x)=\sum_{k=0}^\infty b_{k}x^k$ evaluated at $x_0=-43$ is
$$\sum_{k=0}^\infty \frac{(720 + 1884 k + 1350 k^2 + 960 k^3 + 90 k^4 + 36 k^5)}{k! (6 + k)!}(-43)^k:=\sum_{k=0}^\infty (-1)^k c_k.$$
An elementary computation yields $c_k\ge c_{k+1}$ for $k\ge7$. Hence, $E_k(x_0)$ is an alternating series for $k\ge7$, and for $j=0,1,2$,

\smallskip

$\qquad\qquad \qquad\qquad|E^{(j)}_{30}(x_0)|\le|E_{28}(x_0)|\le|b_{29}| <5\times10^{-18}.$ \end{proof}

\smallskip

The classical Laguerre inequalities (see \cite{CC89}) generalize to a system of inequalities which characterize $\scr L$-$\scr P$, the Laguerre-P\'olya class. 

\begin{thm}[M. L. Patrick  \cite{patrick}, T. Craven and G. Csordas (\cite{CC89}, \cite{CV90}, \cite{CC02})]
\label{lag}
Let $\varphi(x):=\sum_{k=0}^\infty \frac{\gamma_k}{k!}x^k$ be a real entire function. For $x\in\reals$, $n=0,1,2,\ldots$, set
$$L_n(\varphi(x)):=\sum_{k=0}^{2n} \frac{(-1)^{k+n}}{(2n)!}\binom{2n}{k}\varphi^{(k)}(x)\varphi^{(2n-k)}(x).$$
Then $\varphi(x)\in\scr L$-$\scr P$ if and only if
$$L_n(\varphi(x))\ge0 \quad (\text{ for all }x\in\reals, \,n=0,1,2,\ldots).$$
\end{thm}

\smallskip

\begin{thm}
\label{s6}
If $x_0=-43$, then 
\begin{equation}
(f'(x_0))^2-f(x_0)f''(x_0)<0,
\label{eq:lag}
\end{equation}
where $f(x) = S_6[e^x]$. 
\end{thm}

\begin{proof}
With the notation of Lemma \ref{s6error}, $f_{n}(x):=\sum_{k=0}^n a_{k}x^k$, and $E_n(x):=f(x)-f_n(x)$. Using Mathematica 7, for $x_0=-43$, we obtain 
\begin{eqnarray}
& f(x_0) & = f_{30}(x_0)+E_{30}(x_0) \nonumber\\
&& = -5.354465\ldots\times 10^{-2}+E_{30}(x_0),\nonumber\\
\nonumber\\
& f'(x_0)& = f^{(1)}_{30}(x_0)+E_{30}^{(1)}(x_0)\nonumber\\
&& = 7.536322\ldots\times 10^{-5}+E_{30}^{(1)}(x_0), \qquad  \text{and}\nonumber\\
\nonumber\\
& f''(x_0) & =f^{(2)}_{30}(x_0)+E_{30}^{(2)}(x_0)\nonumber\\
&& = -3.954149\ldots\times 10^{-3}+E_{30}^{(2)}(x_0).\nonumber
\end{eqnarray}

Hence,

\medskip

$\quad (f'(x_0))^2-f(x_0)f''(x_0)$

\medskip

$\quad= ( 7.536322\ldots\times 10^{-5}+E_{30}^{(1)}(x_0) )^2 $ 

\medskip

$\quad\quad\; - ( -5.354465\ldots\times 10^{-2}+E_{30}(x_0))(-3.954149\ldots\times 10^{-3}+E_{30}^{(2)}(x_0) ).$

\medskip

By Lemma \ref{s6error}, a calculation show that

\medskip

$\qquad \qquad \qquad  \qquad (f'(x_0))^2-f(x_0)f''(x_0) \;<\; - 2.1 \times 10^{-4}$.
\end{proof}

\medskip

A similar argument, \emph{mutatis mutandis}, establishes the following theorem.

\medskip

\begin{thm}
\label{wts6}
If $x_0=-56$, then
\begin{equation}
(g'(x_0))^2-g(x_0)g''(x_0)<0,
\label{eq:lag2}
\end{equation}
where $g(x)= \widetilde{S}_6[e^x]$. 
\end{thm}

\section{Related results}
\label{s:related}
\medskip

In \cite[Question 3]{Fisk}, S. Fisk raised the following question.

\medskip

\begin{prob}
\label{fisk2}
Let $d\in\nats$, and let $f(x)=\sum_{k=0}^n a_kx^k\in \scr  L\text{-}\scr P_\nats^+$, where $\scr  L\text{-}\scr P_\nats^+$ is defined in~\eqref{eq:p+}. Form
\begin{equation}
\label{eq:fdfinite}
F_{d}[f(x)]:=\sum_{k=0}^n \abs{\begin{array}{ccc} a_k & \ldots & a_{k+d-1}\\ a_{k-1} & \ldots & a_{k+d-2}\\ \vdots & & \vdots\\ a_{k-d+1} & \ldots & a_{k}\\  \end{array}}x^k,\text{ where } a_k=0 \text{ for } k<0\text{ and }k>n.
\end{equation}
Is it true that $F_{d}[f(x)]\in\scr  L\text{-}\scr P_\nats^+$ for all $f(x)\in \scr  L\text{-}\scr P_\nats^+$?
\end{prob}

\medskip

In order to give a partial answer to S. Fisk's question, we introduce the following notation. 

\begin{notn}
\label{mtrx}
For a given sequence of complex numbers $\set{a_k}_{k=0}^\infty$, we consider the infinite matrix
\begin{equation}
\label{eq:infmtx}
M:=\lr{\begin{array}{lllll}
a_0 & a_1 & a_2 & a_3 & \ldots\\
a_{-1}& a_0 & a_1 & a_2  & \ldots\\
a_{-2} & a_{-1} & a_0 & a_1 & \ldots\\
a_{-3} & a_{-2} & a_{-1} & a_0 & \ldots\\
\vdots & \vdots & \vdots & \vdots & \ddots \\
\end{array}}.
\end{equation}
Furthermore, we define the $d\times d$ principal minor, starting at column $k$ of $M$, by
\begin{equation}
\label{eq:det}
D_k^{(d)}:= \disp\det{\lr{a_{k-i+j}}}, \quad \text{for}\quad 0\le i,j \le d-1.
\end{equation}
\end{notn}

\medskip

As an application of P. A. MacMahon's Master Theorem \cite[Section 495]{Mac}, R. P. Stanley \cite[Theorem 18.1]{stanley} proved the following result.

\medskip

\begin{thm}[R. P. Stanley \cite{stanley}]
\label{Mac}
 Let $d, n\in \nats$, $a_k:=\binom{n}{n-k}$, and $a_k:=0$ for $k<0$ and $k>n$. Then for $0\le k \le n$, 
$$D_k^{(d)}= \prod_{j=0}^{d-1}{\frac{\binom{n+j}{k+j}}{\binom{n-k+j}{n-k}}} $$
where $D_k^{(d)}$ is defined in \eqref{eq:det}.
\end{thm}

\medskip

The following composition theorem (see \cite[Theorem 2.4]{CCkluwer}, \cite[\S 16]{marden}, \cite[\S 7]{O63}) will be needed in the sequel.

\begin{thm}[Malo-Schur-Szeg\"o Theorem] 
\label{mss}
Let
$$A(x)=\sum_{k=0}^n \binom{n}{k}a_k x^k, \quad \quad B(x)=\sum_{k=0}^n \binom{n}{k}b_k x^k,$$
and set
$$C(x)=\sum_{k=0}^n \binom{n}{k}a_k b_k x^k.$$

\begin{enumerate}
\item $($\emph{Szeg\"o} \cite{szego22}$)$ If all the zeros of $A(x)$ lie in a circular region $K$, and if $\beta_1,\beta_2,\ldots, \beta_n$ are the zeros of $B(x)$, then every zero of $C(x)$ is of the form $\zeta=-\omega\beta_j$, for some $\omega\in K$.
\item $($\emph{Schur} \cite{schur}$)$ If all the zeros of $A(x)$ lie in a convex region $K$ containing the origin and if the zeros of $B(x)$ lie in the interval $(-1,0)$, then the zeros of $C(x)$ also lie in $K$.
\item If the zeros of $A(x)$ lie in the interval $(-a,a)$ and if the zeros of $B(x)$ lie in the interval $(-b,0)$ (or in $(0,b)$), where $a,b,>0$, then the zeros of $C(x)$ lie in $(-ab,ab)$.
\item $($\emph{Malo} \cite[p.\,29]{O63}, \emph{Schur} \cite{schur}$)$ If the zeros of $p(x)=\sum_{k=0}^\mu a_k x^k$ are all real and if the zeros of $q(x)=\sum_{k=0}^\nu b_kx^k$ are all real and of the same sign, then the zeros of the polynomials $h(x)=\sum_{k=0}^m k!a_kb_kx^k$ and $f(x)=\sum_{k=0}^m a_kb_kx^k$ are all real, where $m=\min(\mu,\nu)$.
\end{enumerate}
\end{thm}

\smallskip

\begin{lem}
\label{bx}
If $d,n\in\nats$ and 
$$B(x):= \sum_{k=0}^n \binom{n}{k}{\frac{\binom{n+d}{k+d}}{\binom{n-k+d}{n-k}}}x^k = \sum_{k=0}^n \binom{n}{k}\br{\frac{(n+d)! d!}{(k+d)! (n-k+d)!}}x^k,$$ 
then $B(x)$ has only real negative zeros.
\end{lem}

\begin{proof}
Two proofs will be given.

\smallskip

\emph{Proof 1.} The numerator in the summand of $B(x)$, $(n+d)! d!$, are fixed constants.  As noted before (cf. Example \ref{xms}), $\set{\frac{1}{(k+d)!}}_{k=0}^\infty$ is a multiplier sequence. By Lemma \ref{rvf}, the sequence $\set{\frac{1}{(n-k+d)!}}_{k=0}^\infty$ (where $ \frac{1}{k!}=0$ for $k<0$) is also a multiplier sequence. By Theorem \ref{PS}, applying these multiplier sequences to $\sum_{k=0}^n\binom{n}{k}x^k$ implies that $B(x)$ has only real negative zeros.

\smallskip

\emph{Proof 2.} K. Driver and K. Jordaan \cite[Theorem 3.2]{dj08} proved that the hypergeometric polynomial $\, _2\!\,F_1(-n,-(n+d); d; x) = B(x)$ has only real negative zeros.
\end{proof}

\smallskip

Using Theorem \ref{Mac}, Theorem \ref{mss}, Lemma \ref{bx}, and Lemma \ref{rvf}, a partial answer to Open Problem \ref{fisk2} is given in the following proposition.

\begin{prop}
\label{binom}  
For $d, n\in \nats$, the polynomial $F_{d}[(1+x)^n] $ has only real negative zeros, where $F_d$ is defined in ~\eqref{eq:fdfinite}.
\end{prop}

\begin{proof} Fix $n\in\nats$. By Theorem \ref{Mac}, $F_{d}[(1+x)^n]=\sum_{k=0}^n \br{\prod_{j=0}^{d-1}{\frac{\binom{n+j}{k+j}}{\binom{n-k+j}{n-k}}}}x^k $. We will complete the proof of the proposition by induction on $d$. 
$$F_{1}[(x+1)^n]=\sum_{k=0}^n \binom{n}{k} x^k = (1+x)^n\in \scr  L\text{-}\scr P_\nats^+.$$ 

\noindent Suppose $A(x):=F_{d}[(1+x)^n]=\sum_{k=0}^n \br{\prod_{j=0}^{d-1}{\frac{\binom{n+j}{k+j}}{\binom{n-k+j}{n-k}}}}x^k $ has only real negative zeros. Consider $B(x)=\sum_{k=0}^n \binom{n}{k}{\frac{\binom{n+d}{k+d}}{\binom{n-k+d}{n-k}}}x^k$ from Lemma \ref{bx}, which has only real negative zeros. By Theorem \ref{mss}, the composition of $A(x)$ and $B(x)$ is
$$C(x)=\sum_{k=0}^n \br{\prod_{j=0}^{d}{\frac{\binom{n+j}{k+j}}{\binom{n-k+j}{n-k}}}}x^k=F_{d+1}(x),$$
which has only real negative zeros.
\end{proof}

\smallskip

A proof analogous to Proposition \ref{binom} yields the following result about hypergeometric polynomials.

\smallskip

\begin{prop}
\label{hypG} 
For a finite subset $P\subseteq\nats$, denote by $|P|$ the number of elements in $P$. Then the hypergeometric polynomial $($see \cite[p.\,73]{rain}$)$

\medskip
\noindent $\disp _{|P|+1}F_{|P|}(-n, -(n+\alpha_1),\ldots, -(n+\alpha_{|P|}); \;\; \alpha_1,\ldots,\alpha_{|P|}; \;\; (-1)^{|P|+1}x)$
$$ =1+\sum_{k=1}^\infty \frac{\disp \prod_{i=1}^{|P|+1}(-(n+\alpha_{i-1}))_k}{\disp \prod_{j=1}^{|P|}(\alpha_{j})_k}\frac{x^k}{k!}  = \sum_{k=0}^n \binom{n}{k} \br{\prod_{\alpha_i\in P\subseteq\nats}{\frac{\binom{n+\alpha_i}{k+\alpha_i}}{\binom{n-k+\alpha_i}{n-k}}}}x^k$$
has only real negative zeros, where $\alpha_0=0$ and $(m)_j=m(m+1)(m+2)\cdots(m+j-1)=\frac{\Gamma(m+j)}{\Gamma(m)}$ is the Pochhammer symbol \cite[p.\,22]{rain}.
\end{prop} 

\begin{notn}
Given a function $f(x)=\sum_{k=0}^\infty a_k x^k\in \scr L\text{-}\scr P^+$, define the associated matrix $M$ as in ~\eqref{eq:infmtx} formed by the sequence $\set{a_k}_{k=0}^\infty$ of coefficients of $f(x)$, where $a_k=0$ for $k<0$. Regard the transformation $F_d$ as a non-linear operator on $\scr L\text{-}\scr P^+$, where
\begin{equation}
\label{eq:fdlp}
F_{d}[f(x)]:=\sum_{k=0}^\infty \abs{\begin{array}{ccc} a_k & \ldots & a_{k+d-1}\\ a_{k-1} & \ldots & a_{k+d-2}\\ \vdots & & \vdots\\ a_{k-d-1} & \ldots & a_{k}\\  \end{array} }x^k \quad (a_k=0 \text{ for } k<0).
\end{equation}
By the Cauchy-Hadamard formula, $F_{d}[f(x)]$ is an entire function.
\end{notn}

The next lemma can be proved by induction. 

\smallskip

\begin{lem}
\label{MacE}
If $d\in \nats$, and the sequence $\set{a_k}_{k=0}^\infty := \set{\frac{1}{k!}}_{k=0}^\infty$, with $a_k=0$ for $k<0$, then 
$$D_k^{(d)}=\prod_{j=0}^{d-1} \frac{j!}{(k+j)!}.$$
where $D_k^{(d)}$ is defined in \eqref{eq:det}.
\end{lem}

\smallskip

Using Lemma \ref{MacE}, the following result is attained.

\begin{prop}
\label{fdex}
For $d\in\nats$, $F_d[e^x] \in \scr L\text{-}\scr P^+$, where $F_d$ is defined in ~\eqref{eq:fdlp}.
\end{prop}

\begin{proof}
Fix $d\in\nats$. 
$$F_{d}[e^x]=\sum_{k=0}^\infty \abs{\begin{array}{ccc} a_k & \ldots & a_{k+d-1}\\ a_{k-1} & \ldots & a_{k+d-2}\\ \vdots & & \vdots\\ a_{k-d-1} & \ldots & a_{k}\\  \end{array} }x^k,$$ 
where $a_k=\frac{1}{k!}$, and  $a_k=0$  for $k<0.$ Then by Lemma \ref{MacE}, 
$$\disp F_{d}[e^x] =\sum_{k=0}^\infty \lr{\prod_{j=0}^{d-1} \frac{j!}{(k+j)!}} x^k.$$
Since $\set{\frac{1}{(k+j)!}}_{k=0}^\infty$is a multiplier sequence for $j=0,1,\ldots,d-1$, $F_{d}[e^x] \in \scr L\text{-}\scr P^+$.
\end{proof}

\section{An application}
\label{s:app}

For a sequence of positive real numbers $\set{a_k}_{k=0}^\infty$, D. K. Dimitrov \cite{dim} defined the \emph{higher order Tur\'an inequalities} as \begin{equation}
\label{eq:dim}
4(a_k^2-a_{k-1}a_{k+1})(a_{k+1}^2- a_{k}a_{k+2}) - (a_ka_{k+1}- a_{k-1}a_{k+2})^2\ge0.
\end{equation}

\medskip

\noindent For a polynomial $\sum_{k=0}^n a_kx^k$, define the non-linear operator $J$ acting on $\scr  L\text{-}\scr P_\nats^+$ (cf. Definition \ref{ms}) by
$$J\br{\sum_{k=0}^n a_kx^k\!}\!:=\!\! \sum_{k=0}^n \! \br{4\!\lr{a_k^2-a_{k-1}a_{k+1}}\!\lr{a_{k+1}^2- a_{k}a_{k+2}} - \!\lr{a_ka_{k+1}- a_{k-1}a_{k+2}}^2}\! x^k, $$
where $a_k=0$ for $k<0$ and $k>n$. The operator $J$ has the following property.

\medskip

\begin{prop}
\label{genT}
If $n\in\nats$, then $J[(1+x)^n]\in \scr  L\text{-}\scr P_\nats^+$, where $\scr  L\text{-}\scr P_\nats^+$ is defined in~\eqref{eq:p+}.
\end{prop}

\begin{proof}

$J[(1+x)^n] $

\medskip

$= \disp \sum_{k=0}^n$ $\!\br{4\!\lr{\!\binom{n}{k}^2\!-\binom{n}{k-1}\binom{n}{k+1}\!}\!\!\lr{\!\binom{n}{k+1}^2\!- \binom{n}{k}\binom{n}{k+2}\!} - \!\lr{\!\binom{n}{k}\binom{n}{k+1}- \binom{n}{k-1}\binom{n}{k+2}\!}^2} x^k$

$ $

$= \disp \lr{4 n! (n+1)! (n+2)!} \sum_{k=0}^n \binom{n}{k} \br{ \frac{x^k}{(k+1)! [(k+2)!]^2 (n-k-1)! [(n-k+1)!]^2}}.$

\medskip

\noindent By Example \ref{xms} and Lemma \ref{rvf}, $\;\set{\frac{1}{(k+1)!}}_{k=0}^\infty\,$, $\;\set{\frac{1}{(k+2)!}}_{k=0}^\infty\,$, $\;\set{\frac{1}{(n-k-1)!}}_{k=0}^\infty\,$, \,and $\;\set{\frac{1}{(n-k+1)!}}_{k=0}^\infty\,$ (where $\frac{1}{m!}=0$ for $m<0$) are multiplier sequences. Thus \mbox{$J[(1+x)^n]\in \scr  L\text{-}\scr P_\nats^+$}.
\end{proof}

\smallskip

\begin{notn}
\label{jlp}
For $f(x)=\sum_{k=0}^\infty a_kx^k \in \scr L\text{-}\scr P^+$, extend the operator $J$ from $\reals[x]$ to $\scr L\text{-}\scr P^+$ as the operator $F_d$ was extended in ~\eqref{eq:fdlp}. Thus
\begin{equation}
\label{eq:jlp}
J[f(x)]:= \sum_{k=0}^\infty \br{4(a_k^2-a_{k-1}a_{k+1})(a_{k+1}^2- a_{k}a_{k+2}) - (a_ka_{k+1}- a_{k-1}a_{k+2})^2}x^k.
\end{equation}
By the Cauchy-Hadamard formula, $J[f(x)]$ is an entire function. 
\end{notn}

\smallskip

\begin{prop}
\label{jex}
If $f(x)=e^x$ in Notation \ref{jlp}, then $J[e^x] \in \scr L\text{-}\scr P^+$, where $J[e^x]$ is defined by ~\eqref{eq:jlp}.
\end{prop}

\begin{proof}
$J[e^x] $

\smallskip

$= \disp \sum_{k=0}^\infty \left(4\br{\lr{\frac{1}{k!}}^2-\frac{1}{(k-1)!}\frac{1}{(k+1)!}}\br{\lr{\frac{1}{(k+1)!}}^2 - \frac{1}{k!}\frac{1}{(k+2)!}} \right. $ 
$$\left. \qquad  \qquad \qquad   \qquad\qquad \qquad \qquad  \qquad \qquad  - \br{\frac{1}{k!}\frac{1}{(k+1)!}- \frac{1}{(k-1)!}\frac{1}{(k+2)!}}^2  \right) x^k$$
$= \disp \sum_{k=0}^\infty \frac{4}{k! (k+1)! [(k+2)!]^2} x^k$

\noindent Since $\set{\frac{1}{(k+j)!}}_{k=0}^\infty$ is a multiplier sequence for $j=0,1,2$, $J[e^x] \in \scr L\text{-}\scr P^+$.
\end{proof}

\section{Questions and An example}
\label{s:qe}

We pose some questions regarding the operator $S_r$ (similar questions could be considered for the operator $\widetilde S_r$).

\smallskip

\begin{prob}
\label{whichr}
Find all $r\in\nats$ such that $S_r[\scr L\text{-}\scr P^+]\subseteq \scr L\text{-}\scr P^+$.
\end{prob}

\begin{prob}
\label{srall}
Characterize the functions  $f(x)\in\scr L\text{-}\scr P^+$ such that $S_r[f(x)]\in \scr L\text{-}\scr P^+\cup\set{0}$ for all $r\in\nats$. 
\end{prob}

\smallskip

The existence of functions in Open Problem \ref{srall} is a consequence of the following notion (see \cite[Section 4]{CC95}).

\begin{defn}
A sequence of positive real numbers $\set{s_k}_{k=0}^\infty$ is called \emph{rapidly decreasing}, if the sequence satisfies the following condition
$$s_k^2\ge\alpha^2s_{k-1}s_{k+1}, \text{ for } \alpha\ge\max \set{2, \frac{\sqrt{2}}{2}\lr{1+\sqrt{1+s_1}}}.$$
A power series $f(x)=\sum_{k=0}^\infty s_k x^k$ whose coefficients form a rapidly decreasing sequence $\set{s_k}_{k=0}^\infty$ belong in $\scr L\text{-}\scr P^+\cup\set{0}$ (see \cite[Section 4]{CC95}, or \cite{hu}).
\end{defn}

\begin{xmpl}
\label{rpfn}
The function $f(x)=\sum_{k=0}^\infty \frac{x^k}{2^{k^2}}\in\scr L\text{-}\scr P^+,$
since the sequence 
$$\disp \set{\frac{1}{2^{k^2}}}_{\!\!k=0}^{\!\!\infty}\!:=\;\set{a_k}_{k=0}^\infty$$
satisfies $a_k^2\ge4a_{k-1}a_{k+1}$ for $k\in\nats$. For $r\in\nats$, define the sequence 
$$\set{t_{k,r}}_{k=0}^\infty:=\set{a_k^2-a_{k-r}a_{k+r}}_{k=0}^\infty.$$
Then the sequence $\set{t_{k,r}}_{k=0}^\infty$ also satisfies the condition $t_{k,r}^2\ge 4t_{k-1,r}t_{k+1,r}$ for $k\in\nats$. Hence $\set{t_{k,r}}_{k=0}^\infty$ is a rapidly decreasing sequence, so $S_r[f(x)] \in\scr L\text{-}\scr P^+$ for all $r\in\nats$.
\end{xmpl}

The following question is related to Open Problem \ref{srall} and Example \ref{rpfn}.

\begin{prob}
\label{rapid}
Let $f(x)=\sum_{k=0}^\infty a_k x^k$ where $\set{a_k}_{k=0}^\infty$ is a rapidly decreasing sequence. If $S_r[f(x)]=\sum_{k=0}^\infty (a_k^2-a_{k-r}a_{k+r}) x^k:=\sum_{k=0}^\infty b_k x^k$, is it true that $\set{b_k}_{k=0}^\infty$ is a rapidly decreasing sequence? 
\end{prob}


\section*{Acknowledgments}
The author is indebted to Dr. George Csordas, Mr. Matthew Chasse, and Mr. Lukasz Grabarek for their careful reading of the manuscript and many helpful suggestions.


\end{document}